\documentclass[reqno,12pt]{amsart} 
\usepackage[top=1.5in,right=1.125in,left=1.125in,bottom=1.5in]{geometry}
\usepackage{amssymb}
\usepackage{amsmath}
\usepackage{color}
\usepackage{graphicx}

\definecolor{red}{rgb}{0.7,0,0}
\definecolor{grey}{RGB}{112,112,112}
\definecolor{blue}{RGB}{034,113,179}
\usepackage[colorlinks=true,citecolor=blue,linkcolor=red,urlcolor=grey,backref=page]{hyperref}

\newcommand{\koniec}{\begin{flushright}  $\Box $ \end{flushright}}
\newtheorem{theo}{Theorem}[section] 
\newtheorem*{theo*}{Theorem}
\newtheorem{prop}[theo]{Proposition}  

\newtheorem{defi}[theo]{Definition}

\theoremstyle{remark}

\newcommand{\ce}{\mathcal{E}}

\newcounter{mnotecount}[section]

\renewcommand{\themnotecount}{\thesection.\arabic{mnotecount}}
\newcommand{\pp}{\mbox{$\scriptsize{\Pi}$}}

\newcommand{\mnote}[1]
{\protect{\stepcounter{mnotecount}}$^{\mbox{\footnotesize
$
\bullet$\themnotecount}}$ \marginpar{
\raggedright\tiny\em
$\!\!\!\!\!\!\,\bullet$\themnotecount: #1} }

\newcommand{\hook}{{\setlength{\unitlength}{11pt}   
                   \begin{picture}(.833,.8)
                   \put(.15,.08){\line(1,0){.35}}
                   \put(.5,.08){\line(0,1){.5}}
                   \end{picture}}}

\newcommand{\Z}{\mathbb{Z}}
\newcommand{\PP}{\mathbb{P}}
\newcommand{\RP}{\mathbb{RP}}
\newcommand{\R}{\mathbb{R}}

\newcommand{\Rho}{\mathrm{P}}

\def\p{\partial}
\def\be{\begin{equation}}
\renewcommand{\d}{d}

\def\ee{\end{equation}}

\def\bea{\begin{eqnarray}}
\def\eea{\end{eqnarray}}

\def\ov{\overline}

\numberwithin{equation}{section}
\begin{document} \date{11 February 2020}
\title[projective and $c$--projective compactifications of 
Einstein metrics]{Some examples of projective and $c$--projective compactifications of 
Einstein metrics}

\author{Maciej Dunajski}
\address{Department of Applied Mathematics and Theoretical Physics\\ 
University of Cambridge\\ Wilberforce Road, Cambridge CB3 0WA, UK.}
\email{m.dunajski@damtp.cam.ac.uk}

\author{A.\ Rod Gover}
\address{Department of Mathematics\\
 The University of Auckland, \\
Private Bag 92019, Auckland 1142, New Zealand}
\email{r.gover@auckland.ac.nz}

\author{Alice Waterhouse}
\address{Department of Applied Mathematics and Theoretical Physics\\ 
University of Cambridge\\ Wilberforce Road, Cambridge CB3 0WA, UK.}
\email{aw592@cam.ac.uk}
\begin{abstract}
We construct several examples of compactifications of Einstein metrics. We show that
the Eguchi--Hanson instanton admits a projective compactification which is non--metric, and that
a metric cone over any (pseudo)--Riemannian manifolds admits a metric projective compactification. 
We construct a para--$c$--projective compactification of neutral signature Einstein metrics canonically defined on certain rank--$n$ affine bundles $M$ over $n$--dimensional manifolds endowed with projective structures.
 \end{abstract}


\maketitle

\section{Introduction}
There are several notions of compactifications of a (pseudo) Riemannian
manifold $(M, g)$.  In a conformal compactification $(\ov{M},\ov{g})$
one has that $\ov{M}$ is a manifold with boundary such that $M$ is the
interior of $\ov{M}$, and there is a defining function $T$ for the
boundary such that the metric $\ov{g}=T^2g$ smoothly extends to the
boundary $\partial M$ of $\ov{M}$  ($T$ being a {\em defining function} for $\partial M$ means that $\partial M=\mathcal{Z}(T):=\{p\in \ov{M}~:~T(p)=0\}$ and $dT$ is nowhere zero on $\partial M$.)  The geodesics of $g$ do not correspond to
geodesics of $\ov{g}$, but the angles are preserved. This kind of
compactification has proven to be useful in studying the causal
structure of space times in general relativity \cite{penrose}, and
quantum field theory \cite{witten}. It also underlies formulating the
boundary conditions \cite{uhlen} of conformally invariant field
equations like the Yang--Mills theory where the curvature decay rate
on $\R^4$ is equivalent to the connection extending to a one--point
compactificaton $\ov{\R}^4=\R^4\cup\{\infty\}=S^4$. In this case the round metric on $S^4$ is conformally equivalent to a flat
Euclidean metric on $\R^4$.
 
The conformal compactification is not available or not natural for many complete metrics that one may want to
compactify, which motivates a search for other compactification mechanisms.
 In a projective compactification of a pseudo--Riemannian manifold \cite{CG0}
the unparametrised geodesics of $(M, g)$ 
smoothly extend to a boundary $\p M$ of the manifold
$\ov{M}=M\cup\p M$ (see \S\ref{section2} for definitions). 
In general
this manifold does not carry a metric, but only an affine connection $\ov{\nabla}$ belonging to the projective equivalence class
containing the Levi--Civita connection of $g$. This kind of compactification is naturally applicable to scattering problems \cite{melrose}.

 There are two related concepts of compactifications: In a
 $c$--projective compactification \cite{CG} of an almost complex
 manifold $(M, J)$ with complex connection $\nabla$ the compactifying
 connection $\ov{\nabla}$ belongs to the $c$--projective equivalence
 class of $\nabla$, i.e. $\nabla$ and $\ov{\nabla}$ preserve $J$, have
 the same torsion, and share the same $J$--planar curves (see
 \S{\ref{section_c}} for definitions). In the para--$c$--projective
 compactification which we shall introduce in \S\ref{section_c} the
 endomorphism $J:TM\rightarrow TM$ squares to identity. The boundary
 $\p M$ acquires a contact structure, and a conformal metric on the
 contact distribution.

This paper is organised as follows. In \S\ref{section2} we shall
review the notion of the projective compactificatiton, and show that
the Eguchi--Hanson gravitational instanton can be projectively
compactified. We shall then introduce the notion of a {\em metric
  projective compactification}, where the connection $\ov{\nabla}$ is
the Levi--Civita connection of some metric $\ov{g}$ on $\ov{M}$. We
shall prove (Theorem \ref{theo1}) that a metric cone over a
(pseudo)--Riemannian manifold admits a metric projective
compactification.  In \S\ref{section_c} and \S\ref{section_4} we shall
construct a para-c-projective compactification of a neutral signature
Einstein metric $g$ defined on a projectivised tractor bundle $M$ of
any projective structure (Theorem \ref{metric} and Theorem
\ref{our_thm}).  The model for this construction will be the
compactification of $M=SL(n+1)/GL(n)$ which corresponds to the flat
projective structure on $\RP^n$.
\subsection*{Acknowledgements}
The work of MD has been partially supported by STFC consolidated grant no. ST/P000681/1. ARG gratefully acknowledges support from the Royal
Society of New Zealand via Marsden Grants 13-UOA-113 and 16-UOA-051.
AW is grateful for support from the Sims Fund. MD acknowledges the hospitality of the University of Auckland
where this work has started, and ARG acknowledges the hospitality of the University of Cambridge where it has finished.
MD is grateful to Dmitri Alekseevsky and Vladimir  Matveev for useful 
correspondence.
\section{Metric and non--metric projective compactifications}
\label{section2}
In this section we shall introduce the concept of a metric projective compactification
of a manifold $M$ with an affine connection $\nabla$, and give some examples of metric
and non-metric projective compactifications.
\begin{defi}
\label{defi1}
An affine connection $\nabla$ on $M$ admits a projective compactification of order $\alpha$ to a manifold with boundary $\overline{M}=M\cup\p M$ if there exists a function $T:\overline{M}\rightarrow \R$ such that $T=0$ is the boundary $\p M\subset \overline{M}$, the differential $dT$ does not vanish on $\p M$, and a projectively equivalent connection $\overline{\nabla}$ on $M$ defined by
\be
\label{connection}
\overline\nabla_X Y=\nabla_X Y+\Upsilon(X)Y+\Upsilon(Y)X
\ee
with
\[
\Upsilon=\frac{dT}{\alpha T}
\]
extends smoothly to $\p M$.
\end{defi}

In \cite{CG0} it was shown that if $\nabla$ is the Levi--Civita connection of a (pseudo)--Riemannian metric $g$ on $M$, then a sufficient condition for a projective compactification of order $\alpha>0$ (such that 
$2/\alpha \in \mathbb{Z}$)   to exist
is that near the boundary $\mathcal{Z}(T)$ the metric $g$ can be put in the form
\be
\label{project_comp}
g=C\frac{dT^2}{T^{4/\alpha}}+\frac{1}{T^{2/\alpha}}h,
\ee
for some $h$
which smoothly extends to the boundary $\mathcal{Z}(T)$ and restricts
to a pseudo-Riemannian metric there, and some constant $C$. Moreover
if $g$ is Ricci--flat, and a projective compactification exists, then
necessarily $\alpha=1$.

In the examples below we shall make use of a stronger notion of {\it metric} projective compactifications
\begin{defi}
A projective compactification from Definition \ref{defi1}
is {\it metric} if $\overline{\nabla}$ is the 
Levi--Civita connection of some (pseudo)--Riemannian metric $\ov{g}$ on $\overline{M}$.
\end{defi}
\subsection{Example: Flat space.} Consider a flat metric on $\R^n$ of the form
\be
\label{g_flat}
g_{\mbox{flat}}=dr^2+r^2 \gamma_{S^{n-1}},
\ee
where $\gamma_{S^{n-1}}$ is the round metric on a sphere $S^{n-1}$ with 
Ricci scalar equal to $(n-1)(n-2)$. 
Setting $T=r^{-1}$ puts $g_{\mbox{flat}}$ in the form (\ref{project_comp}), but the resulting 
connection (\ref{connection}) is not metric\footnote{To show this, compute
the Ricci tensor $\overline{R}$ of (\ref{connection}). If (\ref{connection}) was metric for some metric
$\overline{g}$ then $\overline{R}$ would be a Ricci--tensor of $\overline{g}$, and so 
(as $\overline{g}$ is projectively flat) it would have to be a constant multiple of
$\overline{g}$ by the Beltrami theorem. Computing the Ricci tensor of the metric given by $\overline{R}$ shows that
it is impossible for any $\alpha$.}.
To construct a {\em metric projective--compactification}
consider a defining function given by
\[
T=\frac{1}{\sqrt{r^2+1}}.
\]
The metric $g_{\mbox{flat}}$ takes the form
\[
g_{\mbox{flat}}=\frac{dT^2}{T^4}+\frac{1}{T^2}\Big((1-T^2)\gamma_{S^{n-1}}+\frac{1}{1-T^2}dT^2)
\]
and $h$ reduces to $\gamma_{S^{n-1}}$ on the boundary $\mathcal{Z}(T)$. It can now be verified by direct calculation that
the connection (\ref{connection})  with $\Upsilon=T^{-1}dT$ is the Levi--Civita connection of the metric
\be
\label{metric_new}
\overline{g}=\frac{dr^2}{(1+r^2)^2}+\frac{r^2}{r^2+1}\gamma_{S^{n-1}}
=\frac{dT^2}{1-T^2}+(1-T^2)\gamma_{S^{n-1}}.
\ee
The metric $\overline{g}$ has constant positive curvature,  is defined on an open set
of a round sphere $S^{n}$, and extends to the boundary $T=0$ where it induces the metric
$\gamma_{S^{n-1}}$.

This construction has the following natural geometric interpretation (Figure \ref{fig1}). Consider a central projection $\pi$
from a hemi--sphere ${\mathcal S}\subset S^{n}$ to $\R^n$. If the metric on $S^n$ is
\[
g_{S^{n}}=d\theta^2+\cos{\theta^2}\gamma_{S^{n-1}}
\]
and the inverse map $\pi^{-1}:\R^n\rightarrow {\mathcal S}$ is given by 
$
\cos^2{\theta}={r^2}{(r^2+1)}^{-1}
$
then the pull back of $(\pi^{-1})^*g_{S^{n+1}}$ to $\R^{n+1}$ is given by (\ref{metric_new}).
\begin{center}
\label{fig1}
\includegraphics[width=7cm,height=3cm,angle=0]{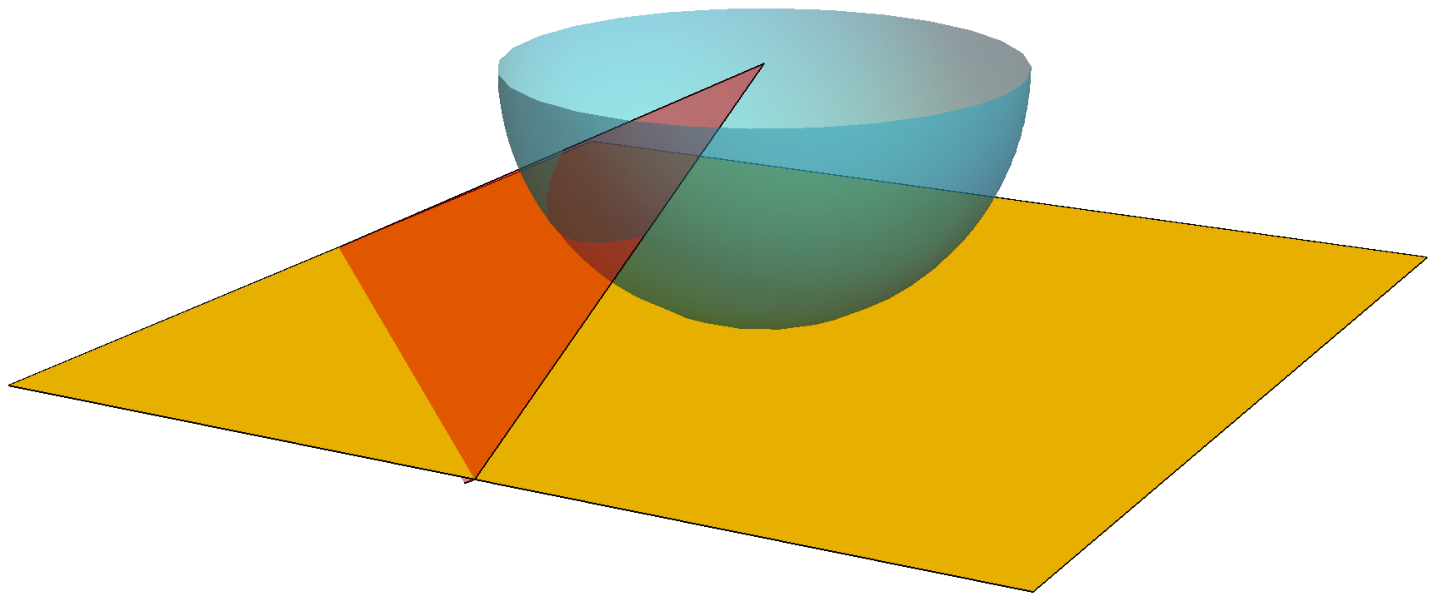}
\end{center}
\subsection{Example: The Eguchi--Hanson manifold}
The Eguchi--Hanson metric \cite{EH} is given 
by\footnote{Left invariant one--forms $\sigma_i, i=1, 2, 3$ 
on the group manifold $SU(2)$ satisfy
\be
\label{cijk}
d\sigma_1+\sigma_2\wedge\sigma_3=0, \quad
d\sigma_2+\sigma_3\wedge\sigma_1=0, \quad d\sigma_3+\sigma_1\wedge\sigma_2=0.
\ee
These one--forms can be represented in terms of Euler angles
by
\[
\sigma_1+i\sigma_2=e^{-i\psi}(d\theta+i\sin{\theta} d\phi), \qquad
\sigma_3=d\psi+\cos{\theta}d\phi,
\]
where to cover $SU(2)=S^3$ we require the ranges
\[
0\leq\theta\leq\pi, \quad 0\leq\phi\leq 2\pi, \quad
0\leq\psi\leq 4\pi.
\]
}
\be
\label{Eguchi_Hanson}
g=\Big(1-\frac{a^4}{r^4}\Big)^{-1}dr^2+\frac{1}{4}r^2
\Big(1-\frac{a^4}{r^4}\Big)\sigma_3^2+\frac{1}{4}r^2(\sigma_1^2+\sigma_2^2).
\ee
The apparent singularity at $r=a$ is removed by allowing
\[
r>a,\qquad 
 0\leq\psi\leq 2\pi,\qquad 0\leq \phi\leq 2\pi,\qquad 0\leq\theta\leq\pi.
\]
Setting $\rho^2=r^2(1-(a/r)^4)$ and expanding the metric
near $r=a$ and fixing $(\theta, \phi)$ gives
$
g\sim (d\rho^2+\rho^2 d\psi^2)/4.
$
In the standard spherical polar coordinates $\psi$ has
a period $4\pi$ on $SU(2)$. In our case
the period of $\psi$ is $2\pi$ to achieve regularity.
Therefore the surfaces
of constant $r$ 
are real projective planes defined by identifying the antipodal points
on the sphere, $\RP^3=S^3/\Z^2$. At large values
of $r$ the metric looks like $\R^4/\Z^2$ rather than Euclidean space.
Thus the Eguchi--Hanson metric is an example
of the  Asymptotically Locally
Euclidean manifold (see e.g. \cite{ADM} for a discussion of this class of manifolds in the context of twistor theory).

To projectively compactify (\ref{Eguchi_Hanson}) introduce the defining 
function
\[
T=\frac{1}{r}
\]
so that (\ref{Eguchi_Hanson}) takes the form (\ref{project_comp}) with
\[
h=\frac{4a^4T^6}{1-a^4T^4}dT^2+(1-a^4T^4)\sigma_3^2+\sigma_1^2+\sigma_2^2.
\]
and the topology of the boundary is $\RP^3$.
The resulting connection (\ref{connection}) is non-metric. One can ask whether there exists another choice of
the defining function which leads to  a metric projective compactification. The answer to that question is negative, as Eguchi--Hanson is geodesically rigid: up to a constant non--zero multiple there exists only one metric in its projective class. This follows from a combination of the following two facts:
(A) Two Ricci--flat metrics in dimension four are projectively equivalent iff they are affinely equivalent (i.e. they share the same Levi--Civita connection). (B)
In the positive signature two affine equivalent Ricci--flat metrics are flat.
See \cite{matveev} for proofs of these facts.
\subsection{Metric cones}
The projective equivalence of (\ref{metric_new})
and (\ref{g_flat}) is an example of the following result of Levi--Civita \cite{LC}
\begin{prop}
\label{proplc}
The metrics 
\[
g=dr^2+f(r)\gamma, \quad\mbox{and}\quad \overline{g}=\frac{1}{(\kappa f(r)+1)^2}
dr^2 +\frac{f(r)}{\kappa f(r)+1}\gamma
\]
are projectively equivalent for any constant $\kappa$. Here $f$ is an
arbitrary function of $r$, and $\gamma$ is an arbitrary $r$--independent metric.
\end{prop}
\noindent
{\bf Proof.}
One way to establish this Proposition is to observe that Levi--Civita connections
of $g$ and $\overline{g}$ are related by (\ref{connection}) with
\[
\Upsilon=-\frac{\kappa}{2(1+\kappa f(r))}\frac{df}{dr} dr.
\]
\koniec
Let $(N, \gamma)$ be a (pseudo) Riemannian manifold of dimension $(n-1)$.
A metric cone of $(N, \gamma)$ is a (pseudo) Riemannian manifold 
$M=N\times\R^+$ with the metric 
\be
\label{cone_metric}
g=dr^2+r^2 \gamma
\ee
We shall use Proposition \ref{proplc} to prove the following
\begin{theo}
\label{theo1}
The metric cone $(M, g)$ given by (\ref{cone_metric}) of a (pseudo)--Riemannian manifold $(N, \gamma)$ admits
a metric projective compactification of order 1.
\end{theo}
\noindent
{\bf Proof.}
 Taking the defining function to be
\[
T=\frac{1}{\sqrt{r^2+1}} 
\]
leads to
\[
{g}=\frac{dT^2}{T^4}+\frac{1}{T^2}h,\quad\mbox{where}\quad
h=\frac{dT^2}{1-T^2}+(1-T^2)\gamma.
\]
The connection $\overline{\nabla}$ given by (\ref{connection})
with $\Upsilon=T^{-1}dT$
extends to the boundary, and is the Levi--Civita connection
of the metric
\begin{eqnarray*}
\overline{g}&=&\frac{dr^2}{(1+r^2)^2}
+\frac{r^2}{1+r^2} \gamma\\
&=&\frac{dT^2}{1-T^2}+(1-T^2)\gamma
\end{eqnarray*}
in agreement with Proposition \ref{proplc}.
The metric $\overline{g}$ extends to the boundary.
\koniec
This construction leads to global examples of metric projective compactifications.
It follows from the general theory \cite{bar} that metrics cones over Einstein manifolds
with non--zero Ricci scalar which admit Killing spinors are manifolds with special holonomy. 
In fact the first examples of non--compact manifolds with exceptional holonomies $G_2$ and $Spin(7)$
were constructed by Bryant \cite{bryant_exeptional} as 
metric cones over the nearly K\"ahler manifold $SU(3)/T^2$ and
the weak $G_2$ holonomy manifold $SO(5)/SO(3)$ respectively.
In view of Theorem \ref{theo1} these metric cones admit metric projective compactifications.

\section{Para-c-projective compactifications of Einstein metrics} 
\label{section_c}

In \cite{CG} the concept of $c$--projective compactification was
defined. It is based on almost $c$--projective geometry \cite{c_proj},
an analogue of projective geometry defined for almost complex
manifolds $(M,J)$ in which the equivalence class of connections
defining the $c$--projective structure must be \textit{complex} and
\textit{minimal}.\footnote{Recall that a connection on an almost
  complex manifold $(M,J)$ is called \textit{complex} if it preserves
  $J$ and \textit{minimal} if the torsion is just the Nijenhuis tensor
  of $J$ up to a constant factor.} Here we review the definition of
$c$--projective compactification, modifying to the ``para'' case where
the para--almost--complex structure squares to $Id$ rather than $-Id$. We
then introduce a class of $2n$--dimensional, neutral signature
Einstein metrics arising from projective structures in dimension $n$,
and show that these admit a compactification which we call
\textit{para}--$c$--projective.

Although $c$--projective compactification is defined for any almost complex manifold, the definition can be applied to pseudo--Riemannian metrics $g$ which are Hermitian with respect to the almost complex structure so long as there exists a connection which preserves both $g$ and $J$ and has minimal torsion. Such Hermitian metrics are said to be \textit{admissible}.  Note that such a connection, if it exists, is uniquely defined, since the conditions that it be complex and minimal determine its torsion. It is thus given by the Levi--Civita connection of $g$ plus a constant multiple of the Nijenhuis tensor of $J$.

We make the following definition in the para--$c$--projective case.

\begin{defi}
\label{defi_1}  Let $(M,g,J)$ be a para--Hermitian manifold, and let $\nabla^L$ be a connection which preserves both $g$ and $J$ and has minimal torsion. The structure $(M, g, J)$
admits a para--$c$--projective compactification to a manifold with boundary $\ov{M}=M\cup\p M$,
if there exists a function $T:\ov{M}\rightarrow \R$ such that 
$\mathcal{Z}(T)$ (the set where $T=0$ on $\ov{M}$) is the boundary
$\p M\subset \ov{M}$, the differential $dT$ does not vanish on $\p M$, and the
connection
\[
\overline{\nabla^L}_X Y={\nabla^L}_X Y+\Upsilon(X)Y+\Upsilon(JX)JX
+\Upsilon(Y)X+\Upsilon(JY)JX, \quad\mbox{where}\quad\Upsilon=\frac{dT}{2T},
\]
extends to $\ov{M}$.
\end{defi}

Note that the para--$c$--projective change of connection $\nabla^L\rightarrow \ov{\nabla^L}$ differs from the $c$--projective case in the signs of some of the terms, to account for the fact that $J$ squares to the $Id$ rather than $-Id$.

It follows easily from this definition that the endomorphism $J$ on $M$ naturally extends to all of $\ov{M}$ by parallel transport with respect to $\ov{\nabla^L}$. It thus defines an almost para--CR structure on the hypersurface distribution $\mathcal{D}$ defined by $\mathcal{D}_x:=T_x\p M \cap J(T_x \p M)$. It can be shown (see Lemma 5 of \cite{CG} and modify to the case $J^2=Id$) that this almost para--CR structure is non--degenerate if and only if for any local defining function $T$ the one--form $\theta=dT\circ J$ restricts to a contact form on $\p M$.

The first main result of \cite{CG} is  Theorem 8 in this reference, which gives a local form for an admissible Hermitian metric which is sufficient for the corresponding $c$--projective structure to be $c$--projectively compact. The theorem is stated below, adapted to the para--$c$--projective case. Note that this includes an assumption that the Nijenhuis tensor $\mathcal{N}$ of $J$ takes so--called \textit{asymptotically tangential values}. This is equivalent to the following statement in index notation:
\be
\label{Nijenhuis_condition}
\Big({\mathcal{N}^{a}}_{bc}\nabla_a T\Big)|_{T=0}=0.  \ee
The
following result arises by a trivial adaption of the arguments in
\cite{CG} for the almost complex case, and so further details may be obtained from that source.

\begin{theo}[\cite{CG}] \label{CGthm}
Let $\ov{M}$ be a smooth manifold with boundary $\p M$ and interior $M$. Let $J$ be an almost para--complex structure on $\ov{M}$, such that $\p M$ is non--degenerate and the Nijenhuis tensor $\mathcal{N}$ of $J$ has asymptotically tangential values. Let $g$ be an admissible pseudo--Riemannian Hermitian metric on $M$. For a local defining function $T$ for the boundary defined on an open 
subset ${\mathcal U}\subset \ov{M}$, put $\theta=dT\circ J$ and, given a non--zero real 
constant $C$, define a Hermitian $\big(\negthinspace\begin{smallmatrix}
0 \\ 2
\end{smallmatrix}\negthinspace\big)$--tensor field $h_{T,C}$ on 
${\mathcal U}\cap M$ by
\[
h_{T,C}:=Tg+\frac{C}{T}(dT^2-\theta^2).
\]
Suppose that for each $x\in\p M$ there is an open neighbourhood 
${\mathcal{U}}$ of $x$ in $\ov{M}$, a local defining function $T$ defined on 
${\mathcal{U}}$, and a non--zero constant $C$ such that
\begin{itemize}
\item $h_{T,C}$ admits a smooth extension to all of $\mathcal{U}$
\item for all vector fields $X,Y$ on $U$ with $dT(Y)=\theta(Y)=0$, the function $h_{T,C}(X,JY)$ approaches $Cd\theta(X,Y)$ at the boundary.
\end{itemize}
Then $g$ is $c$--projectively compact.
\end{theo}
The statement in  Theorem \ref{CGthm} does not depend on the choice of $T$. Different choices of $T$ result in rescalings of the contact form $\theta$ on the boundary by a nowhere vanishing function.
\section{Example: A class of neutral signature Einstein metrics} 
\label{section_4}
\newcommand{\ol}[1]{\overline{#1}} \newcommand{\bp}{\boldsymbol{p}}
\newcommand{\cT}{\mathcal{T}} \newcommand{\cE}{\mathcal{E}} 
In this Section we shall  introduce a class of para--Hermitian metrics, which we will show to be para--$c$--projectively compact by virtue of Theorem \ref{CGthm}.
\begin{defi}
A projective structure $[\nabla]$ on an $n$--dimensional manifold $N$ is an equivalence class of torsion--free affine
connection on $N$ which share the same unparametrised geodesics.  Two connections $\overline{\nabla}$ and $\nabla$
belong to the same projective class if they are related by (\ref{connection}) for some one--form
$\Upsilon$ on $N$.
\end{defi}
Given a projective structure $[\nabla]$ on an $n$--dimensional manifold $N$, there exists  a canonical neutral signature Einstein metric $g$ with non-zero scalar curvature, as well as a symplectic form $\Omega$, on the total space $M$ of a certain rank $n$ affine bundle over $N$. The general construction of this metric and the proof of the projective invariance has been presented in \cite{DM} in dimension four, and
in \cite{DW20} in general dimension. Here we shall give an explicit coordinate description from \cite{DM}. In Theorem \ref{metric} we shall show how to recover the metric and the symplectic form (\ref{main_metric}) from the natural pairing on the co-tractor bundle.

The sections of the affine bundle $M \rightarrow N$ are in one-to-one
correspondence with the $[\nabla]$ representative connections
(see \S 2.4 and \S 3.4 in \cite{DM}), 
and
hence the choice of a representative connection $\nabla \in [\nabla]$
provides a diffeomorphism $T^*N \rightarrow M$. Pulling back the pair
$(g,\Omega)$ with this diffeomorphism gives a pair
$(g_{\nabla},\Omega_{\nabla})$ on $T^*N$ which -- in local coordinates
$(x^i,\xi_i)$ on the cotangent bundle $\nu : T^*N \to N$ -- takes the
form
\begin{align}
\label{main_metric}
g&=\left(d\xi_i-\left(\Gamma_{ij}^k \xi_k-\xi_i\xi_j-\Rho_{(ij)}\right)\d x^i\right)\odot \d x^j,\\
\Omega&=\d \xi_i\wedge \d x^i+\Rho_{[ij]}\d x^i\wedge \d x^j,
\quad\mbox{where}\quad
i, j =1, \dots, n.
\nonumber
 \end{align}
Here $\Gamma^i_{jk}$ denote the Christoffel symbols and $\Rho_{ij}$ is the Schouten tensor of of $\nabla$. 

In \cite{DM} it was shown that the manifold $M$ can be identified with the
complement of
an $\RP^{n-1}$ sub-bundle in the  projectivisation $\mathbb{P}(\cT^*)$ of a certain rank $(n+1)$ vector bundle (the so called co-tractor bundle) $\cT^*$  
over $N$. In the special case where $N=\RP^n$, and $[\nabla]$ is projectively 
flat the manifold $M=SL(n+1, \R)/GL(n, \R)$ can be identified with the projection of $\R^{n+1}\times \R_{n+1}\setminus {\mathcal Z}$, where ${\mathcal Z}$ denotes the set
of incident pairs (point, hyperplane). See \cite{mettler1} for other 
applications of (\ref{main_metric}).

The compactification procedure described in  
the Theorem \ref{our_thm} will, for the model, attach these incident pairs back to $M$, 
and more generally
(in case of a curved projective structure on $N$) will attach
the  $\RP^{n-1}$ sub-bundle of $\mathbb{P}(\cT^*)$. The boundary $\p M$ 
from Definition \ref{defi_1} will play a role of a submanifold manifold
separating two open sets in $\mathbb{P}(\cT^*)$ in the sense described 
in the next sub-Section.

\subsection{The constructions in tractor terms}\label{tractor-ver}
Let the  projective structure $(N, [\nabla])$, be represented by
some torsion free affine connection $\nabla$ on $N$, where   the latter has
dimension at least 2.  Let $\cE(1)\rightarrow N$ be  the line bundle which is the standard
$-2(n+1)$th root 
of the square of the canonical bundle of
$N$ (which, note, is canonically oriented).  For any
vector bundle $\mathcal{B}$ and line bundle $\cE(w)$ we write
$\mathcal{B}(w)$ as a shorthand for $\mathcal{B}\otimes \cE(w) $.

Canonically on  the projective manifold  $(N, [\nabla])$ there is the rank $(n+1)$
cotractor bundle \cite{BEG}
\[
\cT^*\rightarrow N .
\]
This has a composition sequence
\begin{equation}\label{co-euler}
0\to T^*N(1)\to \cT^*\stackrel{X}{\to}\cE(1)\to 0,
\end{equation}
where the map $X\in \Gamma(\cT(1))$ is called the (projective) {\em canonical
  tractor}.
Choosing a connection in $[\nabla]$ determines a splitting of this sequence and so then have
$\cT^*= \cE(1) \oplus  T^*N(1)$, and we can represent an element $V$ of $\cT^*$ as a pair
$(\sigma, \mu)=[V]_\nabla$ (see e.g. \cite{CG0}). Any other connection $\ol{\nabla}$ in $[\nabla]$ is related to $\nabla$ by (\ref{connection}), for some  1-form field  $\Upsilon$ on $N$, and the corresponding transformation 
\begin{equation}\label{ttrans}
  [V]_\nabla=(\sigma, \mu)\mapsto (\sigma, \mu +\sigma \Upsilon)=[V]_{\ol{\nabla}}.
\end{equation}
The main importance of $\cT^*$ is that it admits a canonical
projectively invariant {\em tractor connection} $\nabla^{\cT^*}$ given by
\be
\label{tractor_con}
{\quad{\nabla^{\cT}}_i \left(\begin{array}{c}
\sigma\\ 
\mu_j
\end{array} \right)= 
\left(\begin{array}{c} \nabla_i\sigma-\mu_i \\ 
\nabla_i\mu_j+\Rho_{ij}\sigma
\end{array} \right).}
\ee
We shall now present two variants of the 
construction of \cite{DM}, and then its compactification (to be made precise
in Theorem \ref{our_thm}).
We begin with
compactification of the construction in \cite{DM}. For simplicity let us assume that $N$ is
orientable.

\subsubsection{Compactification by line projectiviation:}
On the total space of $\cT^*$ we pullback
$\pi:\cT^*\to N$ along $\pi$ to get $\pi^*(\cT^*)\to \cT^*$ as a
vector bundle over the total space $\cT^*$. By construction this
bundle has a tautological section $U\in \Gamma (\pi^*(\cT^*))$.  We
also have $\pi^*(\cT(w))$ for any weight $w$, and we
shall write simply $X\in \Gamma(\pi^*(\cT(1)))$ for the pullback to
$\cT^*$ of the canonical tractor $X$ on $N$.

There is  a canonical density $\tau\in \Gamma(\pi^*\cE(1))$ given by
$$
\tau:= X\hook U .
$$

Now define
\be
\label{projection_map}
\kappa: \cT^*\longrightarrow \mathcal{M}:=\mathbb{P}(\cT^*)
\ee
by the fibrewise
projectivisation, and use also $\pi$ for the map to $N$:
$$
\pi:\mathcal{M}\to N.
$$ Note that $\tau$ is homogeneous of degree 1 up the fibres of the
map $\cT^*\to \mathcal{M}$. Thus $\tau$ determines, and is equivalent
to, a section (that we also denote) $\tau$ of a certain density bundle
$\pi^*(\cE(1))\otimes \cE_{\cT^*}(1)$, on $\mathcal{M}$ that for
simplicity we shall denote $\cE(1,1)$.  So $\mathcal{M}$ is stratified according to
whether or not $\tau$ is vanishing, and we write $\mathcal{Z}(\tau)$
to denote, in particular, the zero locus of $\tau$.

To elaborate
on the densities used here, and their generalisation to arbitrary
weights: By $\cE_{\cT^*}(w')$, for $w'\in \mathbb{R}$, we mean the line
bundle on $\mathbb{P}(\cT^*)$ whose sections correspond to functions
$f: \pi^*\cT^* \to\mathbb{R} $ that are homogeneous of degree $w$ in
the fibres of $\pi^*\cT^*\to \mathbb{P}(\cT^*)$. Then for any weight $w$ we also have $\cE(w)$ on $N$ and its pull back to the bundle $\pi^*\cE(w)\to \mathbb{P}(\cT^*)$.
Then
$$
\ce(w,w'):= \pi^*\cE(w) \otimes \cE_{\cT^*}(w').
$$

Using these tools we can recover the metric of \cite{DM}:
\begin{theo}\label{metric} 
  There is a neutral signature metric on $\mathcal{M}\setminus \mathcal{Z}(\tau)$ determined by the canonical pairing of the horizontal and vertical subspaces
of $T(\cT^*)$. This metric is Einstein, with non--zero Ricci scalar, and 
agrees with (\ref{main_metric}).
\end{theo}
\begin{proof}
 Considering first the total space $\cT^*$ and then its tangent
 bundle, note that there is an exact sequence
  \begin{equation}\label{TM}
0\to \pi^* \cT^*\to T(\cT^*)\to \pi^*TN\to 0,
  \end{equation}
  where we have identified $\pi^* \cT^*$ as the vertical sub-bundle of $T(\cT^*)$.
A connection on the vector bundle $\cT^*\to N$ is equivalent to a
splitting of this sequence; a connection identifies $\pi^*TN$ with a distinguished  sub-bundle of horizontal subspaces in 
$ T(\cT^*)$.
Thus, in particular, the projective tractor
connection on $\cT^*\to N$ gives a canonical splitting of the sequence (\ref{TM}).
So we have 
\begin{equation}\label{HV}
T(\cT^*)=  \pi^*TN\oplus \pi^* \cT^* .
\end{equation}

We move now to the total
space of $\mathbb{P} {\cT^*}$, and we note that again the tractor
(equivalently, Cartan) connection determines a splitting of the
tangent bundle $T(\mathbb{P} {\cT^*})$, see \cite{CGH-duke}. From the usual Euler sequence
of projective space (or see (\ref{useful}) in the last Section) it
follows that for $T(\mathbb{P} {\cT^*})$ the second term of the
display (\ref{HV}) is replaced by a quotient of $\pi^* \cT^*(0,1)$.  Indeed, if
we work at a point $p\in \mathbb{P}(\cT^*)$, observe that $\pi^*
\cT^*(0,1)$ has a filtration
\begin{equation}\label{quotient}
0\to \cE (0,0)_p\stackrel{U_p}{\to}  \pi^* \cT^*(0,1)|_p \to  \pi^* \cT^*(0,1)|_p/\langle U_p \rangle \to 0
\end{equation}
where, as usual, $U$ is the canonical section. 
But away from $\mathcal{Z}(\tau )$, we have that  $U$ canonically splits 
the appropriately re-weighted pull back of  the sequence (\ref{co-euler})
$$
0\to \pi^* T^*N(1,1)   \to \pi^* \cT^*(0,1) \stackrel{X/\tau}{\to} \cE(0,0) \to 0 .
$$
This identifies the quotient in (\ref{quotient}), and thus we have canonically
$$
T(\mathbb{P}(\cT^*)\setminus \mathcal{Z}(\tau))=  \pi^*TN\oplus \pi^*T^*N(1,1).
$$
It follows that on  $\mathcal{M}$
there is canonically a metric $\boldsymbol{g}$ and symplectic form $\boldsymbol{\Omega}$ taking values in $\cE(1,1)$, given by
\begin{eqnarray*}
\boldsymbol{g}(w_1,w_2)&=& \frac{1}{2}\Big(
\pp_H(w_1)\hook \pp_V(w_2)+\pp_H(w_2)\hook\pp_V(w_1)\Big) \quad \mbox{and} \\
\boldsymbol{\Omega}(w_1,w_2)&=& \frac{1}{2}\Big(\pp_H(w_1)\hook \pp_V(w_2)-\pp_H(w_2)\hook\pp_V(w_1)\Big)
\end{eqnarray*}
where
$$
\pp_H: T(\mathcal{M}\setminus \mathcal{Z}(\tau))\to \pi^*TN \quad \mbox{and} \quad \pp_V: T(\mathcal{M}\setminus \mathcal{Z}(\tau))\to \pi^* T^*N(1,1)
$$
are the projections.
Then we obtain the metric and symplectic form by
\be
\label{almost_there}
g:=\frac{1}{\tau}\boldsymbol{g} \qquad \mbox{and} \qquad \Omega:=\frac{1}{\tau}\boldsymbol{\Omega} .
\ee
What remains to be done, is to show that (\ref{almost_there}) agrees
with the normal form (\ref{main_metric}) once a trivialisation of
$\cT^*\rightarrow N$ has been chosen.

Let $p\in N$ and let ${\mathcal W}\subset N$ be an open 
neighbourhood of $p$ with 
local coordinates $(x^1, \dots, x^n)$ such that
$T_pN=\mbox{span}(\p/\p x^1, \dots, \p/\p x^n)$. The connection 
(\ref{tractor_con}) gives a splitting of $T(\cT^*)$ into the horizontal and
vertical sub-bundles
\[
T(\cT^*)=H(\cT^*)\oplus V(\cT^*),
\]
as in (\ref{HV}).
To obtain the explicit form of this splitting, let $V_\alpha, \alpha=0, 1, \dots, n$ be components of a local section of $\cT^*$ in the trivialisation over ${\mathcal{W}}$.
Then
\[
\nabla^{\cT^*} V_\beta=d V_\beta-\gamma_{\beta}^{\alpha} V_\alpha,
\]
where $\gamma_{\alpha}^\beta= \gamma_{i\alpha}^\beta dx^i$, and the components
of the co-tractor connection 
$\gamma_{i\alpha}^\beta$  are given in terms of the connection
$\nabla$ on $N$, and its Schouten tensor, and 
can be read--off from (\ref{tractor_con}):
\[
\gamma_{i0}^0=0, \quad \gamma_{i0}^j=\delta_i^j,\quad
\gamma_{ij}^k=\Gamma_{ij}^k, \quad \gamma_{ij}^0=-\Rho_{ij}.
\]
In terms of these components we can write
\[
H(\cT^*)=\mbox{span}
\Big( \frac{\p}{\p x^i}+ {\gamma_{i\alpha}^\beta} V_\beta
\frac{\p}{\p V_{\alpha}}, i=1, \dots, n \Big),
\quad V(\cT^*)=\mbox{span}\Big(\frac{\p}{\p V_{\alpha}}, \alpha=0, 1, 
\dots, n\Big).
\]
Setting $\xi_i=V_i/V_0$, where $\tau=V_0\neq 0$ 
on the complement of 
$\mathcal{Z}(\tau)$, 
  we can compute the push forwards
of these spaces to $\mathbb{P}(\cT^*)\setminus {\mathcal{Z}(\tau)}$:
\[
\kappa_* H(\cT^*)=\mbox{span}\Big(h_i\equiv
\frac{\p}{\p x^i}-
(\Rho_{ij}+\xi_i\xi_j  -\Gamma_{ij}^k\;\xi_k)\frac{\p}{\partial \xi_j}
\Big), \quad \kappa_* V(\cT^*)=\mbox{span}\Big(v^i\equiv\frac{\p}{\p \xi_i}
\Big).
\]
The non--zero components of the  metric (\ref{almost_there}) are given by
\[
g(v^i, h_j)={\delta^i}_j
\]
which indeed agrees with (\ref{main_metric}) which is known to be Einstein
\cite{DM}.
\end{proof}

Next we observe that $\mathbb{P}(\cT^*)\setminus \mathcal{Z}(\tau)$ is an affine bundle modelled on $T^* N$.
The point is that given
  $\nabla$ in the projective class there is a  there is a smooth fibre bundle isomorphism
  \begin{equation}\label{key-id}
\iota : T^* N\to \mathbb{P}(\cT^*)\setminus \mathcal{Z}(\tau).
    \end{equation}
First, given  $\nabla$, we can represent an element $U\in
  \cT^*_p$ ($p\in N$) by the pair $(\tau , \mu) \in \cE (1)_p\oplus
  T_p^*N(1)$,  or, if we choose coordinates on $N$, by collection
\be
\label{tractor_U}
U=(\tau, \mu_i), \quad i=1, \dots, n .
\ee
  Then, dropping the choice $\nabla \in [\nabla]$, $U\in
  \cT^*_p$ is an equivalence class of such pairs by the equivalence
  relation (\ref{ttrans}) that covers the
  equivalence relation between elements of $[\nabla]$. 

  Thus, given $\nabla$, and  from the naturality of all maps, it
  follows that the total space of $T^*N$ can be identified with $\mathbb{P}(\cT^*)\setminus \mathcal{Z}(\tau)$
  by (for each $p\in N$)
\begin{equation}
T_p^*N\ni \xi_i  \mapsto [(1,\xi_i)]=[(\tau,\tau \xi_i)]\in
\mathbb{P}(\cT^*)\setminus \mathcal{Z}(\tau) .
\end{equation}

Thus we may view $\mathcal{M}$ as a
compactification of $T^*N$  and, by construction,
this is a closed manifold iff $N$ is closed.

Note that by this construction it is easily verified that the zero
locus of $\tau$ is a smoothly
embedded hypersurface in $\mathcal{M}$,  and from (\ref{co-euler}) it follows
at once that this may be identified with the total space of the
fibrewise projectivisation $\mathbb{P}(T^*N)$ (which is well known to
have a para-CR structure). 

A feature of this construction is that in each dimension $n$ (of $N$)
either the hypersurface $\mathcal{Z}(\tau)$ (if $n$ odd) is not
orientable, or $\mathcal{M}$ (if $n$ even) is not orientable.

\subsubsection{Compactification by ray-projectiviation:}
\label{ray_sub}
Instead we may follow the construction above but instead define
$\mathcal{M}:= \pi^*(\mathbb{P}_+(\cT^*))$, where $\mathbb{P}_+(\cT^*)$ is
the ray-projectivisation of $\cT^*$ (i.e.\ the fibres of $\cT^*\to
\mathbb{P}_+(\cT^*)$ are isomorphic to $\mathbb{R}_+$). Then the bundles $\ce(w,w')$
should also be defined via ray homogeneity.
In this case, for $N$ orientable,
both $\mathcal{Z}(\tau)$ and $\mathcal{M}$ are orientable, and again
$\mathcal{M}$ is closed iff $N$ is. Now from (\ref{co-euler}) we have that
$\mathcal{Z}(\tau)$ may be identified with the
fibrewise ray-projectivisation $\mathbb{P}_+(T^*N)$.
In this variant of the
construction there are two copies $M_{\pm}$ of $T^*N$ in $\mathcal{M}$
according to the sign of $\tau$. Moreover each of $\mathcal{M}\setminus
M_{\mp}$ is a manifold that is globally a para-c-projective
compactification of $M_{\pm}$ in sense of Theorem \ref{CGthm}.

\subsubsection{Remark on continuing the tractor approach}
It would be possible to achieve our main aims by continuing the
tractor approach.  We will not pursue this here as we want to
emphasise that with little effort the main result now follows directly
form the properties of the metric.  However we sketch just the basic idea: By
our construction above it follows that $\mathcal{M}$ has a canonical
para-c-projective geometry. In the notation as above, $\pi^*{\cT}\oplus
\pi^*{\cT^*}$ is the corresponding para-c-projective tractor bundle
and this has a canonical tractor connection that trivially
extends (in fibre directions) the pull back of the projective
connection (that is available in horizontal directions). The
dual pairing between $\pi^*{\cT}$ and $\pi^*{\cT^*}$ determines a
fibre metric and compatible symplectic form on the bundle
$\pi^*{\cT}\oplus \pi^*{\cT^*}$ and this is obviously preserved by the
connection. What remains is to show that the tractor connection so
constructed satisfies properties that mean that it is {\em normal} in
the sense defined in e.g.\ \cite{CS-book}. With this established then
the main results then follow from the general holonomy theory in
\cite{CGH-duke}.


\subsection{The main theorem}
In the previous subsection \ref{ray_sub} we have presented a candidate 
$\overline{M}\equiv\mathcal{M}\setminus M_{\mp}$ for a para-c-projective
compactification of the Einstein para-Kahler manifold $(M, g, \Omega)$ given 
by (\ref{main_metric}). What remains to be done is to show that near
the boundary ${\mathcal{Z}}(\tau)=0$ of $\overline{M}$ the metric 
(\ref{main_metric}) can be put in the local normal form of Theorem 
\ref{CGthm}.

The endomorphism $J:TM\rightarrow TM$ defined by
$\Omega(X, Y)=g(JX, Y)$ satisfies $J^2=Id$, and the associated Libermann connection $\nabla^{L}$ \cite{Lieb} is given by
\be
\label{lib}
{\nabla^L}_a X_b={\nabla^{\bf g}}_a X_b-{G^c}_{ab} X_c,\quad \mbox{where}\quad
{{G^c}_{ab}}=-{{\Omega}^{cd}}{\nabla^{\bf g}}_d {\Omega_{ab}}
\ee
and $\nabla^{\bf g}$ is the Levi--Civita connection of $g$. This connection is metric, has minimal torsion, and preserves the almost para--complex structure $J$. It thus belongs to a para--$c$--projective equivalence class which we will show to be compactifiable in the sense of Theorem 
\ref{CGthm}.

\begin{theo}
\label{our_thm}
The Einstein almost para--K\"ahler metric $(M, g, \Omega)$ given by 
(\ref{main_metric}) admits a para--$c$--projective compactification
$\overline{M}$. The structure on the
$(2n-1)$--dimensional boundary $\p M\cong \mathbb{P}(T^* N)$ of $\overline{M}$ includes a contact structure together with a conformal structure 
and a para--CR structure
defined on the contact distribution.
\end{theo}
\noindent
{\bf Proof.}
In the proof below we shall explicitly construct the boundary $\p M$ together with the contact structure and the associated conformal structure on the contact distribution. We shall
first deal with the model $M=SL(n+1)/GL(n)$, and then explain how the curvature
of $(N, [\nabla])$ modifies the compactification.

In the model case we can define coordinates $x^i$ on $N=\mathbb{RP}^n$ by taking $X=(1,x^1,\dots,x^n)$, where $(X^0,\dots,X^n)$ are homogeneous coordinates and we are working in an open set where $X^0\neq 0$. The ${x^i}$ are flat coordinates, so the connection components (and hence the Schouten tensor) vanish and (\ref{main_metric}) reduces to
\be
\label{model_metric}
g=d\xi_i\odot dx^i + \xi_i\xi_jdx^i\odot dx^j, \qquad
\Omega=\d \xi_i\wedge \d x^i \quad
\mbox{where}\quad
i, j =1, \dots, n.
\ee
We can relate \cite{DM} the affine coordinates $\xi_i$ on the fibres of $T^*N$ to the tractor coordinates (\ref{tractor_U}) by setting $\xi_i=\mu_i/\tau$ on the complement of the zero locus ${\mathcal Z}(\tau)$ of $\tau$.

Now consider an open set  ${\mathcal U}\subset M$ given
by  $\xi_ix^i>0$, and define the function $T$ on ${\mathcal U}$ by
\be
\label{formula_for_T}
T=\frac{1}{\xi_i x^i}.
\ee
We shall attach a boundary  $\p \mathcal{U}$ to the open set $\mathcal{U}$ 
such that $T$ extends to a function $\overline{T}$ on $\mathcal{U}\cup \p \mathcal{U}$, and
$\overline{T}$ is  the defining function for this boundary.
We then investigate the geometry on $M$ in the limit $T\rightarrow 0$.
It is clear from above that
the zero locus of $\overline{T}$ will be contained in the zero locus $\mathcal{Z}(\tau)$ of $\tau$, and
therefore belongs to the boundary of $\overline{M}$. We will 
use $\overline{T}$ as a defining function for $\overline{M}$ in an open set $\overline{\mathcal{U}}\subset\overline{M}$.
The strategy of the proof is to extend $T$ to a coordinate system on 
$\mathcal{U}$, such that near the boundary the metric $g$ takes a form
as in Theorem \ref{CGthm}.

First define $\theta\in \Lambda^1(\overline{M})$
by 
\be
\label{def_theta}
V\hook \theta=J(V)\hook d T, \quad\mbox{or equivalently}\quad 
\theta_a=\Omega_{ac}g^{bc}{{\nabla}^{\bf g}}_b T, \quad a, b, c=1, \dots, 2n
\ee
where $J$ is the para--complex structure of $(g,  \Omega)$. Using (\ref{model_metric}) this  gives
\[
\theta=2T(1-T)\xi_id x^i-dT.
\]
We need $n$  open sets $U_1, \dots, U_n$ such that $\xi_k\neq 0$ on $U_k$
to cover the zero locus of $T$. Here we chose $k=n$, and adapt
a coordinate system (which we will prove to be Pfaff) given by
\[
(T, Z_1, \dots, Z_{n-1}, X^1, \dots,
 X^{n-1}, Y),
\] 
where $T$ is
given by (\ref{formula_for_T}) and
\[
Z_A=\frac {\xi_A}{\xi_n}, \quad X^A=x^A, \quad Y=x^{n}, \quad\mbox{where}\quad
A=1, \dots, n-1.
\]
We compute
\[
\theta=2(1-T)\frac{dY+Z_AdX^A}{K}-dT, \quad
\xi_n=\frac{1}{KT}, \quad \mbox{where}\quad K\equiv Y+Z_AX^A,
\]
and substitute
\[
\xi_i dx^i=\frac{1}{KT}(dY+Z_AdX^A)
\]
into (\ref{model_metric}). This gives
\be
\label{CG_Form}
g=\frac{\theta^2-dT^2}{4T^2}+\frac{1}{T}h,
\ee
where 
\[
h=\frac{1}{4(1-T)}(\theta^2-dT^2)+\frac{1}{K}\Big(dZ_A\odot dX^A-\frac{1}{2(1-T)}X^A dZ_A\odot(\theta+dT)\Big)
\]
is regular at the boundary $T=0$. This is in agreement with the 
asymptotic form in Theorem \ref{CGthm} (see \cite{CG} for further details).

The restriction $h$ to $\p M$ gives a metric on a distribution ${\mathcal D}=\mbox{Ker} (\theta|_{T=0})$
\be
\label{h000}
\theta|_{T=0}=2\frac{dY+Z_A dX^A}{Y+Z_AX^A}, \quad h_0=\frac{1}{4}{(\theta|_{T=0})}^2+\frac{1}{2(Y+Z_AX^A)}(2dZ_A\odot dX^A-X^AdZ_A\odot(\theta|_{T=0})).
\ee
 Note that $T$ is only defined up to multiplication by a positive function. Changing the defining function in this way results in a conformal rescalling of $\theta|_{T=0}$, thus the metric on the contact distribution is also defined up to an overall conformal scale. We shall choose the scale so that
the contact form is given  by $\theta_0\equiv K\theta|_{T=0}$ on $T(\p M)$,
with the metric on ${\mathcal D}$ given by
\be
\label{on_distri}
h_{\mathcal D}=dZ_A\odot dX^A.
\ee

We now move on to deal with the
curved case where the metric on $M$ is given by 
(\ref{main_metric}).
The coordinate system $(T, Z_A, X^A, Y)$ is as above, and
the one--form $\theta$ in (\ref{def_theta}) is given by
\[
\theta=2T(1-T)\xi_idx^i-dT+2T^2(\Rho_{ij}-\Gamma_{ij}^k\xi_k)x^idx^j,
\]
or in the $(T, Z_A, X^A, Y)$ coordinates,
\[
\begin{split}
\theta=\ 2&(1-T)\frac{Z_AdX^A+dY}{K} - dT \\
+& 2T^2\Bigg[\bigg(\Rho_{AB}-\frac{\Gamma^C_{AB}Z_C+\Gamma^n_{AB}}{TK}\bigg)X^AdX^B 
+\bigg(\Rho_{nB}-\frac{\Gamma^C_{nB}Z_C+\Gamma^n_{nB}}{TK}\bigg)YdX^B \\
+& \bigg(\Rho_{An}-\frac{\Gamma^C_{An}Z_C+\Gamma^n_{An}}{TK}\bigg)X^AdY 
+\bigg(\Rho_{nn}-\frac{\Gamma^C_{nn}Z_C+\Gamma^n_{nn}}{TK}\bigg)YdY\Bigg].
\end{split}
\]

Guided by the formula (\ref{CG_Form}) we define
\[
h=Tg-\frac{1}{4T}(\theta^2-dT^2),
\]
which we find to be
\be
\begin{split}
h=&
\frac{1}{4(1-T)}(\theta^2-dT^2)+\frac{1}{K}\Big(dZ_A\odot dX^A-\frac{1}{2(1-T)}
X^A dZ_A\odot (\theta+dT)\Big)\\
-&\frac{1}{K}\Big(
(\Gamma_{AB}^CZ_C+\Gamma_{AB}^n)dX^A\odot dX^B+
(\Gamma_{nn}^CZ_C+\Gamma_{nn}^n)dY\odot dY+
2(\Gamma_{An}^CZ_C+\Gamma_{An}^n)dX^A\odot dY\Big)\\
+&T(\Rho_{AB}dX^A\odot dX^B+2\Rho_{An}dX^A\odot dY+\Rho_{nn}dY\odot dY).
\end{split}
\ee
This is  smooth as $T\rightarrow 0$.

Restricting $h$ to $T=0$ yields a metric which differs from
(\ref{h000}) by the curved contribution given by the components of the  connection, but not the Schouten tensor. Substituting $dY=K\theta|_{T=0}/2-Z_AdX^A$, disregarding the terms involving $\theta|_{T=0}$ in $h$, and conformally rescalling by 
$K$ yields the metric
\begin{eqnarray}
\label{met_th}
h_{\mathcal D}&=&(dZ_A-\Theta_{AB}dX^B)\odot dX^A,\quad
\mbox{where}\\
\Theta_{AB}&=&\Gamma_{AB}^CZ_C+\Gamma_{AB}^n+
(\Gamma_{nn}^CZ_C+\Gamma_{nn}^n)Z_AZ_B-
2(\Gamma_{An}^CZ_C+\Gamma_{An}^n)Z_B\nonumber
\end{eqnarray}
defined on the contact distribution ${\mathcal D}=\mbox{Ker}(\theta_0)$, 
where $\theta_0=2(dY+Z_AdX^A)$.

We now invoke Theorem \ref{CGthm},  verifying
by explicit computation that the remaining two conditions are satisfied. The first of these conditions is that the metric $h_{\mathcal{D}}$ is compatible with the  
Levi--form of the almost para--CR structure induced on $\p M$ by $J$, i.e.
\be
\label{boundary_compatibility}
h_{\mathcal D}(X, Y)=d\theta_0(JX, Y), \quad\mbox{for}\quad X\in\mathcal{D}.
\ee
The second is that the Nijenhuis tensor takes asymptotically tangential values, i.e. that (\ref{Nijenhuis_condition}) is satisfied.

Both of these follow from computing the para--complex structure $J$ in the $(T, Z_A, X^A, Y)$ coordinates. We find
\be
\begin{split}
\label{J_T=0}
J|_{T=0}=\ &-\frac{\p}{\p X^A}\otimes dX^A +\frac{\p}{\p Y}\otimes dY + \frac{\p}{\p Z_A} \otimes dZ_A + \frac{\p}{\p T}\otimes dT \\
&-\frac{Z_B}{K}\frac{\p}{\p T}\otimes dX^B - \frac{1}{K}\frac{\p}{\p T}\otimes dY \\
&-\big(\Gamma^D_{AB}Z_D+\Gamma^n_{AB}\big)\frac{\p}{\p Z_A} \otimes dX^B + \big(\Gamma^D_{nB}Z_D+\Gamma^n_{nB}\big)Z_C\frac{\p}{\p Z_C} \otimes dX^B \\
&-\big(\Gamma^D_{An}Z_D+\Gamma^n_{An}\big)\frac{\p}{\p Z_A} \otimes dY
+\big(\Gamma^D_{nn}Z_D+\Gamma^n_{nn}\big)Z_C\frac{\p}{\p Z_C} \otimes dY.
\end{split}
\ee
Restricting to vectors in $\mathcal{D}$ amounts to substituting $dY=\theta_0/2-Z_AdX^A$ and disregarding the terms involving $\theta_0$ as above, so that
\[
\begin{split}
J|_{\mathcal{D}}= &-\frac{\p}{\p X^A}\otimes dX^A +Z_A\frac{\p}{\p Y}\otimes dX^A + \frac{\p}{\p Z_A} \otimes dZ_A + \frac{\p}{\p T}\otimes dT \\
&-\frac{2Z_B}{K}\frac{\p}{\p T}\otimes dX^B -\Theta_{AB}\frac{\p}{\p Z_A}\otimes dX^B
\end{split}
\]
and (\ref{boundary_compatibility}) is satisfied.

For the Nijenhuis condition, we use the formula
\[
\mathcal{N}^a_{bc}=J^d_{\ [b}\p_{|d|}J^a_{\ c]}-J^d_{\ [b}\p_{c]}J^a_{\ d}.
\]
Note that we need only consider components of this with $a=T$, and thus only need to work with the $\p/\p T$ components of $J$ to find the terms which look like $\p J$. This is a one--form which we shall call $J^{(T)}$ and find to be
\[
\begin{split}
J^{(T)}=&\bigg(-\frac{Z_B}{K} + \frac{T[2Z_B + (\Gamma^D_{AB}Z_D+\Gamma^n_{AB})X^A + (\Gamma^D_{nB}Z_D+\Gamma^n_{nB})Y]}{K} - T^2[\Rho_{AB}X^A+\Rho_{nB}Y]\bigg)dX^B \\
&\bigg(-\frac{1}{K} + \frac{T[2 + (\Gamma^D_{An}Z_D+\Gamma^n_{An})X^A + (\Gamma^D_{nn}Z_D+\Gamma^n_{nn})Y]}{K} - T^2[\Rho_{An}X^A+\Rho_{nn}Y]\bigg)dY.
\end{split}
\]
Note that this agrees with (\ref{J_T=0}) when $T=0$. We use it to calculate ${\mathcal{N}^{a}}_{bc}\nabla_a T$, dropping terms which vanish when $T=0$ to verify (\ref{Nijenhuis_condition}).

\koniec

\subsection{Two--dimensional projective structures}
In the case if $n=2$ the coordinates on $\p M$ are $(X, Y, Z)$,  and (\ref{met_th}) yields
\[
h_{\mathcal D}=dZ\odot dX-[\Gamma_{11}^2+(\Gamma_{11}^1-2\Gamma_{12}^2)Z+(\Gamma_{22}^2-2\Gamma_{12}^2)Z^2+
\Gamma_{22}^1Z^3]dX\odot dX,
\]
which is transparently invariant under the projective changes 
\[
\Gamma_{ij}^k\longrightarrow \Gamma_{ij}^k+\delta^k_i\Upsilon_j+\delta^k_j\Upsilon_i
\]
of $\nabla$.
In the  two-dimensional case the 
projective
structures $(N, [\nabla])$ are equivalent to second order ODEs which are cubic in
the first derivatives (see, e.g. \cite{BDE})
\begin{equation}
\label{ODE}
\frac{d^2 Y}{d X^2}=\Gamma^1_{22}\Big(\frac{d Y}{d X}\Big)^3
+(2\Gamma^1_{12}-\Gamma^2_{22})\Big(\frac{d Y}{d X}\Big)^2
+(\Gamma^1_{11}-2\Gamma^2_{12})\Big(\frac{d Y}{d X}\Big)-
\Gamma^2_{11},
\end{equation}
where the integral curves of (\ref{ODE}) are the unparametrised geodesics of $\nabla$. 
The integral curves $C$ of $(\ref{ODE})$ are integral submanifolds
of a  differential
ideal ${\mathcal I}=<\theta_0, \theta_1>$, where
\[
\theta_0=dY+ZdX, \quad \theta_1=dZ-\Big(\Gamma_{11}^2+(\Gamma_{11}^1-2\Gamma_{12}^2)Z+(\Gamma_{22}^2-2\Gamma_{12}^1)Z^2+
\Gamma_{22}^1Z^3\Big)dX
\]
are one--forms on a three--dimensional manifold $B=\PP(T^*N)$ with local coordinates $(X, Y, Z)$. If $f:C\rightarrow B$ is an immersion, then $f^*(\theta_0)=0, f^*(\theta_1)=0$ is equivalent
to (\ref{ODE}) as long as $\theta_2\equiv dX$ does not vanish. In terms of these three one--forms
the contact structure, and the metric on the contact distribution are given by
$
\theta_0,  h_{\mathcal D}=\theta_1\odot\theta_2.
$

\subsection{The model via an orbit decomposition}

In this section we describe here the flat (in the sense of parabolic
geometries) model \cite{CDT13, DM} of our construction in tractor
terms.

The flat projective structure on $N=\RP^n$ gives rise to 
the neutral signature para--K\"ahler Einstein metric on $M=SL(n+1)/GL(n)$
\be
\label{DM_metric}
g=d\xi_i\odot dx^i+(\xi_i dx^i)^2, \quad \Omega= d\xi_i\wedge dx^i, \quad\mbox{where}\quad
i, j, \dots =1, \dots, n.
\ee
In \cite{DM}, \S 7.1 it was explained how this homogeneous model
corresponds to the projectivised co-tractor bundle of $\RP^n$, with
an $\RP_{n-1}$ removed from each $\RP_n$ fiber. This $\RP_{n-1}$
corresponds to incident pairs of points and hyperplanes in $\R^{n+1}\times\R_{n+1}$.

Here we shall instead take $N$ to be the sphere $S^n$ with its standard
projective structure as this is orientable in all dimensions and, more
importantly, on this (double cover of $\mathbb{R}\mathbb{P}^n$) the
tractor bundle is trivial, and this simplifies the discussion.  The
underlying space of the (compactified) model of dimension $2n$ is
$S^n\times S_n$ where both $S_n$ and $S^n$ denote spheres that are
dual as we shall explain.

Consider first two vector spaces each isomorphic to $\mathbb{R}^{n+1}$:
$$
V\cong \mathbb{R}^{n+1} \qquad \mbox W \cong \mathbb{R}^{n+1}
$$
and view each as a representation space for an $SL(n+1,\mathbb{R})$
action.  So $G:= SL(V)\times SL(W)$ acts on $V\times W$. (Note that we may
wlog consider $V$ and $W$ as respectively the $\pm 1$ eigenspaces of
the single vector space $\mathbb{V}:=V\oplus W$ equipped with a
$\mathbb{J}$ s.t. $\mathbb{J}^2=1$.)

Now the action of $SL(V)$ descends to a transitive action on the ray
projectivisation $\mathbb{P}_+(V)$ and similarly $SL(W)$ acts
transitively on $\mathbb{P}_+(W)$. Thus  $G:= SL(V)\times SL(W)$ acts transitively on the manifold
$$
{\mathcal M}:= \mathbb{P}_+(V) \times \mathbb{P}_+(W).
$$
We can represent an element of $\mathcal{M}$ in terms of pairs of homogeneous coordinates
$([Y],[Z])$ where $0\neq Y\in V$ and $0\neq Z\in W$.

Note that as a smooth manifold $\mathcal{M}=S^n\times S^n$, but as a homogeneous manifold it is
$$
G/P = \big( SL(V)/P_X \big)\times \big( SL(W)/P_U \big)
$$
where $P_X$ (resp.\ $P_U$) is the parabolic subgroup in $SL(V)$
that stabilises a point $[X]$ in $\mathbb{P}_+(V)$ (resp.\ $[U] \in \mathbb{P}_+(W)$
), and $P$ is the group product $P_X\times P_W$ which itself is a
parabolic subgroup of the semisimple group $G$.

Now introduce an additional structure which breaks the $G$
symmetry. 
Namely we fix an isomorphism
$$
I:W\to V^*
$$
where $V^*$ denotes the dual space to $V$. The subgroup $H\cong SL(n+1,\mathbb{R})$ of $G$
that fixes this may be identified with $SL(V)$ which acts on a pair
$(Y,Z)\in V\times V^*$ by the defining representation and on the first
factor and by the dual representation on the second factor.

Given this structure we may now (suppress $I$ and) view ${\mathcal{M}}$ as
consisting of pairs $([X],[U])$ where $0\neq X\in V$ and $0\neq U\in
V^*$. That is 
$$
{\mathcal{M}}= \mathbb{P}_+(V) \times \mathbb{P}_+(V^*).
$$

This is useful as follows: Each element $[U]$ in
$\mathbb{P}_+(V^*)$ determines an oriented  hyperplane in $V$ and each $[X]\in
\mathbb{P}_+(V)$ an oriented line in $V$.  So now we consider the $H$
action on $M$. This has two open orbits and a closed orbit. The last
is the incidence space 
$$
\mathcal{Z}=\{ ([X],[U])\in \mathcal{M} \mid U(X)=0 \} 
$$
which sits as smooth orientable separating hypersurface in $\mathcal{M}$. Then there are the open orbits
$$
M_+=\{ ([X],[U])\in \mathcal{M} \mid U(X)>0 \} \qquad \mbox{and} \qquad
M_-=\{ ([X],[U])\in \mathcal{M} \mid U(X)<0 \}.
$$
We may think of $\mathcal{Z}$ as the `boundary' (at infinity) for the open orbits $M_\pm$.

We now describe the geometries on the orbits. The claim is that there
are Einstein metrics in $M_\pm$, while $\mathcal{Z}$ is well known as
the model for so-called contact Langrangian (or sometimes called
para--CR) geometry, this is a real analogue of hypersurface type CR
geometry.

First observe that $N_V:=\mathbb{P}_+(V)$
is the flat model of projective
geometry. So in particular we have
$$
0\to \ce_V(-1)\stackrel{X}{\to}\cT_V\to TN_V(-1)\to 0
$$
where $\cT_V$ is the projective tractor bundle on $N_V$ and $X$ is the
tautological section of $\cT(1)$, which coincides with the canonical tractor.
Similarly there a sequence on
$N^W:= \mathbb{P}_+(V^*)$
\begin{equation}\label{useful}
0\to \ce^W(-1)\stackrel{U}{\to}\cT^W \to TN^W(-1)\to 0 .
\end{equation}

There is a natural tractor bundle $\mathcal{T}:=\cT_V\oplus \cT^W $ on $M$. 
Where $X$ and $U$ are not incident this induces a metric on $M$ as
follows. Observe that, at a point $([X],[U])$ where $X\hook U \neq 0$, the  tractor field  $U$ splits the first sequence by $\nu\in \Gamma (\ce(-1,0))$ defined by
$$
\nu:=U/\tau
$$
with $\tau:=X\hook U$ (and where we have used an obvious weight
notation). This follows as $X\hook \nu=1$. Similarly
$$
x:=X/\tau \in \Gamma (\ce(0,-1))
$$
splits the second short exact sequence because $x \hook U=1$. Thus we
obtain a neutral signature metric on $TN_V\oplus TN^W$ by these
two steps: First, using  these splittings yields a bundle
monomorphism
$$
TN_V(-1,0)\oplus TN^W(0,-1) \to \cT_V\oplus \cT^W .
$$ Second, this gives a symmetric form $\boldsymbol{g}$ and symplectic form
$\boldsymbol{\Omega}$ on $TN_V(-1,0)\oplus TN^W(0,-1)$ by then using
the canonical metric and symplectic form on $\cT_V\oplus \cT^W $ given
by the duality of $\cT_V$ and $ \cT^W$. Thus $\boldsymbol{g}\in \Gamma
(S^2T^*M(1,1))$ and $\boldsymbol{\Omega}\in \Gamma
(\Lambda^2T^*M(1,1))$. Then set
$$ g:=\frac{1}{\tau}\boldsymbol{g} \qquad \mbox{and} \qquad  \Omega:=\frac{1}{\tau}\boldsymbol{\Omega}.
$$
The metric $g$ is easily seen to have neutral signature.  It is
Einstein because the tractor metric on $\mathcal{T}$ is parallel for the
tractor connection (see \cite{CGH-duke} for the analogous c-projective
case). The tractor connecction arises from the usual parallel transport on the
vector space $V\oplus V^*$ viewed as an affine manifold.

\end{document}